\documentclass[a4paper,12pt]{amsart}

\usepackage{amsthm,amssymb,amsmath,amscd}
\usepackage[mathscr]{euscript}
\usepackage{textcomp}

\theoremstyle{definition}
\newtheorem{definition}{Definition}%[section]

\theoremstyle{plain}
\newtheorem{theorem}[definition]{Theorem}
\newtheorem{lemma}[definition]{Lemma}
\newtheorem{corollary}[definition]{Corollary}

\theoremstyle{remark}

\newcommand{\E}{\mathrm{E}}
\newcommand{\Var}{\mathrm{Var}}
\newcommand{\eps}{\epsilon}
\newcommand{\Ht}{\mathcal{H}_3}

\title[]
{Random triangular groups at density $1/3$}

\author{Sylwia Antoniuk}

\address{Adam Mickiewicz University,
Faculty of Mathematics and Computer Science
ul.~Umultowska 87,
61-614 Pozna\'n, Poland}

\email{\tt antoniuk@amu.edu.pl}

\author{Tomasz \L{u}czak}

\address{Adam Mickiewicz University,
Faculty of Mathematics and Computer Science
ul.~Umultowska 87,
61-614 Pozna\'n, Poland}

\email{\tt tomasz@amu.edu.pl}

\author{Jacek \'{S}wi\c{a}tkowski}

\address{Instytut Matematyczny,
Uniwersytet Wroc{\l}awski,
pl. Grunwaldzki 2/4,
50-384 Wroc{\l}aw,
Poland}

\email{swiatkow@math.uni.wroc.pl}

\thanks{Tomasz {\L}uczak is partially 
supported by NCN grant 2012/06/A/ST1/00261, while 
Jacek \'Swi\c atkowski is partially
supported by the Polish Ministry of Science and Higher
Education (MNiSW), Grant N201 541738.}

\keywords {random group, Khazdan's property, spectral gap, random graph}

\subjclass[2010]{
Primary:  20P05; %Probabilistic methods in group theory
secondary:
20F05, %Generators, relations, and presentations
05C80.% random graphs
}

\date{August 27th, 2013}

\begin{document}

\begin{abstract}
Let $\Gamma(n,p)$ denote the binomial model of 
a random triangular group.
We show that there exist constants $c, C > 0$ such 
that if $p \le {c}/{n^{2}}$, 
then a.a.s. $\Gamma(n,p)$ 
is  free  and if $p \ge C{\log n}/{n^2}$ 
then a.a.s.  $\Gamma(n,p)$ 
has Kazhdan's property (T). 
Furthermore, we show that there exist constants 
$C',c' > 0$ such that 
if ${C'}/{n^2} \le p \le c'{\log n}/{n^2}$, 
then a.a.s. $\Gamma(n,p)$ is 
neither free nor has Kazhdan's property (T).
\end{abstract}

\maketitle

\section{Introduction}

Let $\langle S | R \rangle$ be a group presentation with 
a set of generators $S$ and a set of relations $R$. We consider the 
following binomial model $\Gamma(n,p)$ of a random group 
which is very similar to the model introduced by \.{Z}uk 
in \cite{Z2003} (see also comments in sections \ref{sec3} and 
\ref{sec4} below). Here $S$ consists of $n$ 
generators, while $R$ consists of relations taken independently 
at random with probability $p$ among all cyclically reduced 
words of length three, i.e. relations are of the form $abc$, 
where $a\neq b^{-1}$, $b\neq c^{-1}$ and $c\neq a^{-1}$. 
Presentations obtained in this way are called \textit{triangular presentations}, and groups induced by them 
are called \textit{triangular groups}.

In the paper we study the asymptotic behavior of the random triangular
group $\Gamma(n,p)$ as $n\rightarrow\infty$. For 
a given function $p=p(n)$ we say that  $\Gamma(n,p)$ has a
given property a.a.s. (asymptotically almost surely) if the 
probability of $\Gamma(n,p)$ having this property tends to 1 
as $n\rightarrow \infty$. 
From \.Zuk's result~\cite{Z2003} (see also \cite{KK2012}) it follows that for every $\eps>0$ 
a.a.s.\ $\Gamma(n,p)$ is free provided $p\le n^{-2-\eps}$, and 
it has a.a.s.\ Kazhdan's property (T) whenever $p\ge n^{-2+\eps}$.
This phenomenon is known as 
the phase transition at density $1/3$ for triangular
random groups. We study the behavior at density $1/3$
more carefully. In particular,
we show that there exist 
constants $c, C>0$ such that if $p \le {c}/{n^2}$ then a.a.s. 
the random group in $\Gamma(n,p)$ is a free group, 
while for $p \ge {C\log n}/{n^2}$ a.a.s. 
it has Kazhdan's property (T). 
What is more interesting, we identify a range 
for parameter $p$ for which a.a.s. the random group 
in $\Gamma(n,p)$ is neither free nor has property (T). 
The main results 
of this paper are stated as the following three theorems.

\begin{theorem}\label{th:01}
There exists a constant $c > 0$ such that if $p \le {c}/{n^2}$, 
then a.a.s. $\Gamma(n,p)$ is a free group.
\end{theorem}

\begin{theorem}\label{th:02}
There exist constants $C', c' > 0$ such that 
if ${C'}/{n^2} \le p \le c'{\log n}/{n^2}$, 
then a.a.s. $\Gamma(n,p)$ is neither a free group nor has Kazhdan's property (T).
\end{theorem}

\begin{theorem}\label{th:03}
There exists a constant $C > 0$ such that 
if $p \ge C{\log n}/{n^2}$, 
then a.a.s. $\Gamma(n,p)$ has Kazhdan's property (T).
\end{theorem}

\section{Proof of Theorem \ref{th:01}}

We shall deduce Theorem \ref{th:01} from the following Lemma.

\begin{lemma}\label{lemma:2.1}
There exists a constant $c>0$ such that for $p\le c/n^2$ a.a.s. 
the random group $\Gamma(n,p)=\langle S|R\rangle$ 
has the following property:
\begin{quote}
for every nonempty subset $R'\subseteq R$ 
of relations there exists $a\in S$ and $r\in R'$ 
such that neither $a$ nor $a^{-1}$ appears in any relation 
$t\in R'\setminus \{r\}$ and precisely one letter 
in $r$ belongs to the set $\{a,a^{-1}\}$.     
\end{quote}
\end{lemma}

Before we prove the above result let us observe that it 
implies Theorem~\ref{th:01}.

\begin{proof}[Proof of Theorem~\ref{th:01}]
Indeed, if the random group $\Gamma(n,p)=\langle S|R\rangle$
has the property described in the assertion of Lemma~\ref{lemma:2.1},
then we can eliminate generators one by one each time decreasing 
the size of both the set of generators and the set of relations by one.
Eventually we end up with a free group with $|S|-|R|$ generators.
\end{proof}

We shall deduce  Lemma~\ref{lemma:2.1} from the corresponding result
on random hypergraphs. To this end let us 
partition all relations from $R$ into three different types.
 
\begin{itemize}
 \item Relations of \textit{type 1} are relations of the form $aaa$.
 \item  Relations of \textit{type 2} are  relations of 
the form $aab$, $aba$ and $baa$, 
where $a\neq b$. For such a relation 
$b$ is called a pivotal term. 

 \item Relations of \textit{type 3} are all the remaining relations, 
that is relations of the form $abc$, where  
$a\neq b$, $b\neq c$ and $c\neq a$. 
In this case each element of the relation is pivotal.
\end{itemize}

Now let us  introduce an auxiliary random hypergraph
(in fact, multi-hypergraph) 
$\mathcal{H} = \mathcal{H}(\Gamma(n,p))$.
The set of vertices of $\mathcal{H}$ consists of all 
generators of $\Gamma(n,p)$, i.e. it coincides with the set $S$. 
Every relation in 
$\Gamma(n,p)$ generates a hyperedge in $\mathcal{H}$. 
For relation $r$ of type 3 we generate a 3-edge 
$E=\{a, b, c\}$ consisting of three generators of $S$
such that $r$ is build from the letters of $E\cup E^{-1}$. 
If $r$ is a relation of type 2, 
then we generate a 2-edge $E = \{a, b\}$ such that $r$ 
is build from the letters of 
$E\cup E^{-1}$ and mark an appropriate element as the pivotal 
for this edge. Finally, 
if $r$ is a relation of type 1, i.e. we have either $r=aaa$ or 
$r=a^{-1}a^{-1}a^{-1}$ for some $a\in S$,  
then we put the 1-edge $\{a\}$ in $\mathcal{H}$.

Our aim is to prove the following lemma which immediately 
implies  Lemma \ref{lemma:2.1}.

\begin{lemma}\label{lemma:2.2}
There exists a constant $c>0$ such that for $p\le c/n^2$ the
following holds.  A.a.s. 
\begin{quote}
 for every nonempty subset $F$ of hyperedges of 
$\mathcal{H}=\mathcal{H}(\Gamma(n,p))$ there
exists a vertex $v$ of  $\mathcal{H}$  and a hyperedge $E\in F$ 
such that $E$ is the only hyperedge from $F$ containing $v$
and, moreover, $v$ is pivotal for $E$. 
\end{quote}
\end{lemma}

\begin{proof}[Proof of Lemma \ref{lemma:2.2}]
It is enough to show that the lemma is true for some $p = c/n^2$, 
where $c>0$ is a constant to be chosen later. 

Observe first that the probability that there is a type~1 edge 
in $\mathcal{H}$ tends to 0 as $n\rightarrow\infty$. 
Indeed, let $X$ denote the random variable which counts 
such edges in $\mathcal{H}$. Then
$$\mathbb{P}(\mathcal{H} \text{ has a 1-edge}) \leq \E X=2np 
\le  {2c/n}\to 0.$$
Thus, we may and shall assume that $\mathcal{H}$ contains  
only edges of type either 2 or 3.

Our further argument is based on the fact that in the 3-uniform 
random hypergraph in which the expected number of 
edges is $an$, where $a$ is a small positive constant, 
a.a.s. all components are of size $O(\log n)$. 
More specifically, we shall use 
the following  special case of a 
theorem of Schmidt-Pruzan and Shamir \cite{S-PS1985}.

\begin{theorem}[Schmidt-Pruzan, Shamir \cite{S-PS1985}]\label{th:S-PS}
Let $\mathcal{H}(n,\rho)$ be the random hypergraph with vertex set 
$[n]=\{1,2,\dots,n\}$, 
in which each subset of $[n]$ with three vertices appears 
independently with probability $\rho$.
Moreover, let $2\rho\binom n2\le C$, 
where $C$ is a constant such that $C<1$.  
Then, there exists a constant $K_C$ such that with probability 
$1-o(n^{-2})$ the largest connected component 
in $\mathcal{H}$ contains less than $K_C\log n$ vertices. 
\end{theorem}

Now let $\Ht$ denote the subgraph of $\mathcal{H}$ 
which consists of all its edges of size three. For each such edge there are 
48 different relations which can generate it.
Therefore, due to Theorem~\ref{th:S-PS}, each component in $\Ht$ has fewer 
than $K\log n$ vertices for some constant $K$ depending only on $c$ and provided $c < 1/48$. 
We show first that each  connected subgraph of 
$\Ht$, other than an isolated vertex,   contains at least 
two vertices belonging to exactly one edge of $\Ht$.
Indeed, let us assume that it is not the case and by 
$X$ we denote the number of non-trivial connected subgraphs on $k$ vertices, 
$4\le k\le K\log n$,
contained in $\Ht$ in which all but at most one vertex belongs to 
at least two edges. Each such subgraph has at least 
$\lceil(2k-1)/3\rceil\ge 3k/5$ edges. Thus, 
instead of the random variable $X$ we consider another random variable
$Y$ which counts graphs on $k$ vertices, $4\le k\le K\log n$, 
with exactly $\lceil 3k/5\rceil$ edges.
Let $\rho$ be the probability that there is an edge in $\Ht$ containing any given three vertices. Obviously, 
$\rho \leq 48p$. Furthermore, the probability that given $k$ vertices span a subgraph of $\Ht$ with 
 $\lceil 3k/5\rceil $ edges is bounded from above by
$$ \binom{\binom{k}{3}} 
{ \lceil 3k/5\rceil }\rho^{\lceil 3k/5\rceil }
\leq (ek^2 /3)^{\lceil 3k/5\rceil }
(48p)^{\lceil 3k/5\rceil }\le (16ek^2p)^{ 3k/5 }.$$

Hence for the expectation 
$\E Y$ of $Y$ we have
\begin{align*}
 \E Y & \leq  \sum_{k=4}^{K\log n}\binom{n}{ k}(16ek^2p)^{ 3k/5} %\nonumber \\ & 
\leq  \sum_{k=4}^{K\log n} \left( \frac{en}{k}\right)^k 
\left(\frac{16ek^3c}{n^2} \right)^{3k/5 } 
  \nonumber \\ 
& \leq  \sum_{k=4}^{K\log n}
 \left(\frac{16^3 e^{8} c^3 k^4}{n}\right)^{k/5} 
%\nonumber \\ & 
\leq K\log n \left(\frac{16^3 e^{8} c^3 (K\log n)^4}{n}\right)^{{4}/{5}} \to 0.
\end{align*}

Since clearly
$$\Pr(X>0)\leq \Pr(Y>0)\leq \E Y,$$
 a.a.s. in each non-trivial connected subgraph of $\Ht$ 
at least two vertices belong to exactly one edge.

As for edges of size two in $\mathcal H$ we shall show that  
 a.a.s.\ each component of $\Ht$ shares a vertex 
with at most one such edge. 
To this end let $Z$ count components of $\Ht$ 
(including trivial ones) 
for which it is not the case. Then, 

\begin{equation*}
 \E Z \leq n (K\log n)^2 n^2 (24p)^2 \le 24^2 K^2 c^2(\log n)^2/n \to 0,
\end{equation*}
and so a.a.s.\ each component of $\Ht$ intersects at most one 
edge of $\mathcal{H}$ of size two; in particular, a.a.s. all edges 
of size two form a matching.

It is easy to see that from the above two statements 
the assertion of Lemma~\ref{lemma:2.2} easily follows.
Indeed, take any subset $F$ of hyperedges of $\mathcal{H}$,
and let $F'$ be the subhypergraph of $F$ which 
consists of edges of size three.
If $F'$ is not empty, then a.a.s. each of its non-trivial
components contains at least two vertices of degree one, and at most 
one such vertex belongs in $F$ to an edge of size two.
Thus, a.a.s. there is at least one vertex in $F$ 
 which belongs to precisely one edge in $F$ and this edge is of size three. Now let us suppose that $F'$ is empty, i.e.
all  edges of $F$ are of size two. Then they a.a.s. form 
a matching and it is enough  to take the pivotal vertex of 
one of these edges. 
\end{proof}

\begin{proof}[Proof of Lemma \ref{lemma:2.1}]
Lemma~\ref{lemma:2.1} is a straightforward consequence of \break
Lemma~\ref{lemma:2.2}.
\end{proof}

\section{Proof of Theorem \ref{th:02}}\label{sec3}

We begin with recalling some rather well known results concerning
triangle random groups. We also briefly comment on the
arguments used to verify them,
as explicit proofs seem not to be present in the literature.

We start with a few preliminary comments. In his paper \cite{Z2003}, \.{Z}uk
studied triangle groups $\Gamma(n,t)$, where $n$ denotes
the cardinality of a generating set $S$ and $t$ denotes the cardinality
of a set $R$ of relations
chosen uniformly at random out of all subsets of cardinality $t$ of the set
of all cyclically reduced words of length three
over the alphabet $S\cup S^{-1}$. \.{Z}uk used the so called density
approach to asymptotic properties of random groups, in which we let
$n\to\infty$ and we keep $t\sim n^{3d+o(1)}$, where $d\in[0,1]$ is 
the (constant) density parameter. 
Let us remark also that, as it is  well known from the general theory of
random structures, if $t$ is close to the expected number of 
relations in $\Gamma(n,p)$, i.e. $t=(8+o(1))n^3p$, then 
the asymptotic behavior  of $\Gamma(n,t)$ and $\Gamma(n,p)$ is very similar 
for a large class of properties, 
including all considered in this paper (see, for instance, 
\cite{JLR}, Chapter 1.4).

%\begin{equation}\label{eq1}

%o(n)\cdot n^{3(d-1)}\le  p \le \omega(n)\cdot n^{3(d-1)} 

%\end{equation}

%then, as $n\to\infty$, the asymptotic behavior, with respect to monotone properties, 

%of random groups $\Gamma(n,p)$ and $\Gamma(n,t)$

%at density $d$ is the same.

The following result, very useful in our further developments, 
is essentially due to Ollivier. We follow the notation of \cite{[O1]} concerning van 
Kampen diagrams $D$. In particular, $|D|$ denotes the number of 2-cells in $D$,
and $|\partial D|$ is the boundary length of $D$ (i.e. the length of the word
for which $D$ is a van Kampen diagram).

\begin{lemma}\label{lemma3.1}
Let $\Gamma(n,p)$ be the triangle random group such that 
$p=n^{3(d-1)+o(1)}$
for some $d<1/2$. Then 
for any $\epsilon>0$, a.a.s.
all reduced van Kampen diagrams $D$ for  $\Gamma(n,p)$
satisfy the isoperimetric inequality
$$
|\partial D|\ge 3(1-2d-\epsilon)|D|.
$$ 
\end{lemma}

\begin{proof}
The proof consists of two steps. First, one shows that for any
fixed number $K$ the required inequality holds a.a.s. 
for all reduced van Kampen diagrams $D$ with $|D|\le K$. Next, one uses
a propagation argument to conclude the full statement. 

In \cite{[O1]} an analog of the first step above is proved
in the context of random groups corresponding to Gromov's density model.
More precisely, it is shown that a.a.s. any reduced
van Kampen diagram with $|D|\le K$ for a density $d$ random group with relations of
length $L\to\infty$ satisfies
$$
|\partial D|\ge L(1-2d-\epsilon)|D|.
$$
The same argument applies to density $d$ triangle groups and, in view
of the remarks at the beginning of this section, yields the first
required step.

The appropriate propagation argument for the second step is Theorem 8 of \cite{[O2]}
(mentioned also as Theorem 60 in \cite{[O1]}). It applies 
directly to triangle random groups and, in view of the first step, 
concludes the proof.
\end{proof}

\medskip

Recall that,
given a presentation $P$ of the form 
$\langle a_1,\dots,a_n|r_1,\dots,r_t  \rangle$,
the {\it presentation complex} $C_P$ is a 2-dimensional cell complex 
with a single vertex $v_0$,
with edges (being oriented loops attached to $v_0$) 
corresponding to the generators
$a_1,\dots,a_n$, and  2-cells corresponding to relations 
$r_1,\dots,r_t$ attached
to the 1-skeleton accordingly. The group $\Gamma$ 
given by the presentation $P$
is (canonically isomorphic to) the fundamental group of $C_P$.  
Lemma \ref{lemma3.1} can be used to show our next result which essentially belongs to Gromov.
An outline of its proof  (in a slightly different setting) 
is given in \cite{[O3]} (the last two paragraphs of Section 2); compare also
Subsection I.3.b in \cite{[O1]}.

\begin{lemma}\label{lemma3.2}
Let $\Gamma=\Gamma(n,p)$ be the triangle random group 
such that $p=n^{3(d-1)+o(1)}$ for some $d<{1/2}$, and let $P$ denote its presentation. 
Then a.a.s. the presentation complex $C_P$
is aspherical. In particular, $C_P$ is a classifying 
space for $\Gamma$. 
\end{lemma}

Lemma \ref{lemma3.2} has the following corollary. 
%part (2) of which is justified by [Br], p.187.

\begin{corollary}\label{corollary3.3}
Let $\Gamma=\Gamma(n,p)$ be the triangle random group 
such that 
$p=n^{3(d-1)+o(1)}$
for some $d<{1/2}$, and let $P$ denote its presentation. 
Then a.a.s. the Euler characteristic of $\Gamma$ is given by $\chi(\Gamma):=\chi(C_P)=1-n+t$.
Moreover, $\Gamma$ is torsion free. 
\end{corollary}

Lemma \ref{lemma3.1} has also the following less known consequence.

\begin{corollary}\label{corollary3.4}
Let $\Gamma=\Gamma(n,p)$ be the triangle random group 
such that 
$p=n^{3(d-1)+o(1)}$
for some $d<{4/9}$, and let $P=\langle S|R\rangle$ 
denote its presentation. 
Then a.a.s. every generator $s\in S$ is nontrivial in $\Gamma$. 
\end{corollary}

\begin{proof}
Applying Lemma \ref{lemma3.1} to any density $d<{4}/{9}$ (and using sufficiently
small $\epsilon$) we get the inequality $|\partial D|>|D|/3$. 
Now, trivialization of a generator implies existence of a reduced 
van Kampen diagram $D$ with $|\partial D|=1$. 
By the above inequality we get that such a diagram consists of
$|D|<3$ cells. Furthermore, any such diagram $D$ (i.e. with $|\partial D|=1$) 
has odd number of cells.
However, this is impossible because the relations in $R$ have length three and are cyclically reduced.
%This concludes the proof.
\end{proof}

In the proof of Theorem \ref{th:02}  
we need the following two simple probabilistic facts.

\begin{lemma}\label{lemma4}
If $p \ge {3}/{n^2}$ then a.a.s. 
$\Gamma(n,p)$ has at least $4n$ relations. 
\end{lemma}

\begin{proof}
It is enough to estimate the number of relations which have three different elements. Let $X$ be the random variable which 
counts  such relations in $\Gamma(n,p)$, 
where $p=3/{n^2}$. 
There are $N=(8+o(1))\binom{n}{3}$ different relations of this type and each of them is chosen 
independently with probability $p$. 
Therefore, $X$ has the binomial distribution 
$B(N,p)$ and so, since $p\to 0$, 
we have 
$$\Var X=Np(1-p)=(1-p) \E X = (1+o(1))\E X,$$
where 
$$\E X=Np=(8+o(1))\binom{n}{3}\frac{3}{n^2}=(4+o(1))n\to\infty, $$
so the assertion follows from Chebyshev's inequality.
\end{proof}

\begin{lemma}\label{lemma5}
If $p \le {\log n}/({25n^2})$ then a.a.s. there exists 
a generator $s$ such that neither $s$ nor $s^{-1}$ belongs 
to any relation in $\Gamma(n,p)=\langle S|R\rangle$.
\end{lemma}

\begin{proof}
Let $Y$ count the generators $s$ such that
neither $s$ nor $s^{-1}$ belongs to any relation 
in $\Gamma(n,p)$.
For a given generator $s$ there are 
$48\binom{n-1}{ 2} + 24(n-1) + 2 = a n^2$ 
different relations which contain either $s$ or $s^{-1}$, where 
$a=24+o(1)$. Therefore
$$\E Y = n(1-p)^{an^2}\ge n^{1/50}\to\infty.$$
Furthermore, it is easy to check that $\Var Y=(1+o(1))\E Y$,
so the assertion follows from Chebyshev's inequality.
\end{proof}

\begin{proof}[Proof of Theorem~\ref{th:02}]
In view of Corollary \ref{corollary3.3}, it follows from Lemma \ref{lemma4} that for $p\ge3n^{-2}$
a.a.s. the Euler characteristic $\chi(\Gamma(n,p))$ of $\Gamma(n,p)$
is positive. Since any free group 
%(including the infinite cyclic group) 
has non-positive Euler characteristic,
it follows that a.a.s. $\Gamma(n,p)$ is not free.

Now, if $p\le{\log n/(25n^2)}$, Lemma \ref{lemma5} asserts that a.a.s. there is
a generator $s\in S$ such that neither $s$ nor its inverse $s^{-1}$
appears in any relation from $R$. By Corollary \ref{corollary3.4}, a.a.s.\ all generators 
from $S$ are nontrivial in $\Gamma(n,p)$. Thus, a.a.s.\ $\Gamma(n,p)$ 
splits nontrivially as the free product
$\Gamma(n,p)=\langle s\rangle*\langle S\setminus\{ s\}\rangle$.
Consequently, a.a.s.\ $\Gamma(n,p)$ does not have Kazhdan's 
property (T),
which completes the proof.
\end{proof}

\section{Proof of Theorem \ref{th:03}}\label{sec4}

We prove that the random group $\Gamma(n,p)$ has a.a.s. Kazhdan's property (T) using spectral properties of a special 
random graph 
associated with $\Gamma(n,p)$. We begin with recalling few notions from spectral graph theory.

Let $G=(V,E)$ be a multigraph. We denote by $A=A(G)$ the \textit{adjacency matrix} of $G$, that is $A = (a_{vw})_{v,w\in V}$, 
where $a_{vw}$ is the number of edges between $v$ and $w$. 
Next, let $d_G(v)$ denote the degree of vertex $v$ 
in $G$ and let $\bar{d}$ be the average degree of a vertex in $G$. In what follows we use the term graph instead of multigraph 
keeping in mind that we allow graphs to have multiple edges.

The \textit{normalized Laplacian} of graph $G$ is a 
symmetric matrix $\mathcal{L}(G) = (b_{vw})_{v,w\in V}$, where
$$b_{vw} = \left\{ \begin{array}{ll}
                   1, & \text{if } v=w \text{ and } d_G(v) > 0, \\
                   -a_{vw} / \sqrt{d_G(v) d_G(w)}, & \text{if } \{v,w\}\in G, \\
                   0, & \text{otherwise}.
                  \end{array} \right.
$$
Let $0=\lambda_1 \leq \lambda_2 \leq \ldots \leq \lambda_n$ denote the eigenvalues of $\mathcal{L}(G)$. 
It is easy to see that the 
Laplacian is positive semidefinite and so all its eigenvalues 
are non-negative. The eigenvector corresponding to the smallest eigenvalue 
$\lambda_1 = 0$ has entries $(\sqrt{d_G(v)})_{v\in V}$. 
Furthermore, the remaining eigenvalues 
are bounded above by $2$. The value of 
$\lambda_2$ is called the \textit{spectral gap} of $\mathcal{L}(G)$.

If $G$ has no isolated vertices, then taking $D=D(G)$ to be the diagonal degree matrix with entries $d_{vv} = d_G(v)$, 
we can express the normalized Laplacian of $G$ as $$\mathcal{L}(G) = I - D^{-1/2} A D^{-1/2}.$$ 
Thus,  $1-\lambda_i$, $i=1,2,\dots,n$, are the eigenvalues 
of the matrix $I - \mathcal{L}(G) = D^{-1/2}AD^{-1/2}$
with the same eigenvectors as the eigenvectors for $\lambda_i$ 
for $\mathcal{L}(G)$. 
Therefore, to find the spectrum of $\mathcal{L}(G)$ it is 
enough to study the spectrum of $D^{-1/2}AD^{-1/2}$ instead. 
One way to do this is to use the well-known Courant-Fischer principle 
(cf. \cite{C1920}, \cite{F1905}).

\begin{theorem}[Courant-Fischer Formula]
Let $M$ be a $n\times n$ symmetric matrix with eigenvalues $\lambda_1 \leq \ldots \leq \lambda_n$ and the corresponding 
eigenvectors $v_1, \ldots, v_n$. For $1\leq k \leq n-1$ let $R_k$ denote the span of $v_{k+1}, \dots, v_n$, $R_n = \{0\}$. 
Let $R_k^{\bot}$ denote the orthogonal complement of $R_k$. Then
$$\lambda_k = \max_{\substack{\|x\|=1 \\ x\in R_k^{\bot}}} \langle Mx, x\rangle.$$
\end{theorem}

For a triangular group presentation $P = \langle S | R \rangle$
we define a graph $L=L_P$, called the \textit{link graph} of 
$P$, in the following way. 
The set of vertices of $L$ consists of all generators from $S$ together with their formal inverses $S^{-1}$. Furthermore, 
every relation $abc$ present in $R$ generates three edges $\{a,b^{-1}\}$, $\{b,c^{-1}\}$ and $\{c,a^{-1}\}$ in $L$.

The key ingredient of our argument is the following result 
of \.{Z}uk who showed that 
if the spectral gap of $\mathcal{L}(L)$ is sufficiently large, then the group has property (T) 
(see Proposition 6 on p. 661 of \cite{Z2003}, where $L'(S)$ coincides with our
$L_P$,
and where $\lambda_1$ denotes the same as our $\lambda_2$).

\begin{theorem}[\.{Z}uk \cite{Z2003}]\label{th:1}
Let $\Gamma$ be a group generated by a finite group presentation $P = \langle S | T \rangle$ and let $L=L_P$. 
If $L$ is connected and $\lambda_2[\mathcal{L}(L)] > 1/2$, then $\Gamma$ has Kazhdan's property (T).
\end{theorem}

Recall that \.{Z}uk studied the model $\Gamma(n,t)$ of a random triangular group in which we keep $t \sim n^{3d+o(1)}$. 
To simplify calculations in estimating the spectral gap of the Laplacian he considered the so called \textit{permutation model} 
of random triangular groups
which is more convenient to work with than the triangular model $\Gamma(n,t)$ because in this model the corresponding 
link graph is regular. \.{Z}uk proved that a.a.s. the Laplacian of the link graph in the permutation model 
has large spectral gap. He also stated that the same holds for the triangular model $\Gamma(n,t)$ for $d > 1/3$. 
However, he did not fully justified this statement. Recently, Kotowski and Kotowski \cite{KK2012} 
showed how to modify \.Zuk's argument for the
permutation model to make it 
work for  the triangular model $\Gamma(n,t)$ as well,
completing the proof of the following theorem.

\begin{theorem}[\.{Z}uk \cite{Z2003}, Kotowski and Kotowski \cite{KK2012}]\label{th:2}
If $d>1/3$, then a.a.s. $\Gamma(n,t)$ has property~(T). 
\end{theorem}

Since property (T) is a monotone property, 
Theorem \ref{th:2} implies that for any $\epsilon > 0$ 
and $p=\Omega(n^{\epsilon-2})$, a.a.s. the random group $\Gamma(n,p)$ has property (T). Here we give a stronger result, 
namely we show that the Laplacian of the link graph of $\Gamma(n,p)$ has a large spectral gap provided that 
$p \geq C{\log n}/{n^2}$ for a sufficiently large constant $C>0$.

\begin{theorem}\label{th:3}
Let $L$ be the link graph of $\Gamma(n,p)$.
There exists $C>0$ such that if $p \geq C{\log n}/{n^2}$, then a.a.s. 
$\lambda_2[\mathcal{L}(L)] > 1/2$.
\end{theorem}

Theorem \ref{th:03} is an immediate consequence of Theorem \ref{th:1} and Theorem~\ref{th:3}. In the remaining part 
of this section we give the proof of Theorem \ref{th:3}.

The main idea of our argument comes from \.Zuk who used 
a similar approach in \cite{Z2003}.
First, we  divide the graph $L$ into three random graphs 
$L_1$, $L_2$ and $L_3$ 
which shall behave in a similar way as the Erd\H{o}s-R\'{e}nyi random graph $G(2n,\rho)$, for some appropriately chosen $\rho$. 
 We partition  $L$ into graphs $L_i$ in the following way. 
The three graphs $L_i$ have the same vertex set as $L$, that is the set $S\cup S^{-1}$ of all generators together 
with their formal inverses. For every relation $abc\in R$ we place the edge $\{a,b^{-1}\}$ in $L_1$, $\{b,c^{-1}\}$ in $L_2$ and 
$\{c,a^{-1}\}$ in $L_3$. Therefore in graphs $L_i$ every edge appears independently from other edges. Note however that 
between any two vertices 
we can have multiple edges, in particular there can be up to $4n-4$ such edges between any pair of vertices $a$ and $b$, where 
$a^{-1}\neq b$. Furthermore, unlike in \.Zuk's original proof, 
our graphs are not regular, which is the main small difficulty in our argument.
However, it turns out that adding a small correction to a graph which has the 
degree sequence concentrated around a particular value does not affect much the size of the spectral gap of the Laplacian.

\begin{lemma}\label{lemma:2}
Let $0 < \epsilon < 1$ and let $G$ be a connected graph on $n$ vertices
 such that for any vertex $v$ in $G$, 
$|d_G(v) - d| \leq \epsilon d$. Let $H$ be a graph on the same vertex set and such that $d_H(v) \leq \epsilon d$
for any vertex $v$ in $H$. Then 
$$\lambda_{n-1}[I-\mathcal{L}(G\cup H)] \leq  \lambda_{n-1}[I-\mathcal{L}(G)] + 
\frac{\epsilon}{1-\epsilon}$$ 
or equivalently 
$$\lambda_2[\mathcal{L}(G\cup H)] \geq \lambda_2[\mathcal{L}(G)] - 
\frac{\epsilon}{1-\epsilon}\,.$$ 
\end{lemma}

\begin{proof}
Let $A_G=A(G)$, $D_G=D(G)$, $A_H=A(H)$ and $D_H=D(H)$.
All entries in $A_H$ are nonnegative and thus the spectral norm of $A_H$ can be bounded 
from above by the maximum sum of entries in a row, which is equal to the maximum degree of $H$. Therefore we infer that 
$$\|A_H\| \leq \epsilon d.$$

Notice that $A = A_G + A_H$ is the adjacency matrix of the graph $G\cup H$ and $D = D_G + D_H$ is its degree matrix. 
Since $D$ is a diagonal matrix such that $d(1-\epsilon) \leq d_{ii} \leq d(1 + 2\epsilon)$,
%, $\|D\| \geq d(1 -\epsilon)$ and $\|D^{-1}\| \leq \frac{1}{d(1-\epsilon)}$. Also, 
%$\frac{1}{(d(1 + 2\epsilon))^{1/2}} \leq (davicii_{ii})^{-1/2} \leq \frac{1}{(d(1 -\epsilon ))^{1/2}}$ and 
$$\|D^{-1/2}\| \leq (d(1 -\epsilon))^{-1/2}.$$ 

Moreover, $D_G^{1/2}D^{-1/2}$ is also a diagonal matrix with all entries nonnegative and bounded above by 1. Thus,
$$\|D_G^{1/2}D^{-1/2}\| \leq 1.$$

Let $X$ be the eigenvector of $D^{-1/2}AD^{-1/2}$ corresponding to the largest eigenvalue 
$\lambda_n[D^{-1/2}AD^{-1/2}] = 1$ and $Y$
be the eigenvector of $D_G^{-1/2}A_G D_G^{-1/2}$ corresponding to the largest eigenvalue 
$\lambda_n[D_G^{-1/2}A_G D_G^{-1/2}]=1$. Then 
$$X_i = \sqrt{d_G(i) + d_H(i)} \text{ and } Y_i = \sqrt{d_G(i)}.$$

Furthermore, notice that since $D_G^{1/2}D^{-1/2} X = Y$, if $x\bot X$ and $y=D_G^{1/2}D^{-1/2}x$, then $y \bot Y$.

We can now estimate $\lambda_{n-1}[I-\mathcal{L}(G\cup H)]$ using Courant-Fischer formula, 
Cauchy-Schwarz inequality, and the fact, that for 
diagonal matrix $M$ we have $\langle Mx, y \rangle = \langle x, My\rangle$ for any vectors $x$ and $y$:
\begin{align*}
 \lambda_{n-1}&[I-\mathcal{L}(G\cup H)]  =  \lambda_{n-1}[D^{-1/2}A D^{-1/2}] %\nonumber 
\\ 
 & =  \max_{x\bot X, \|x\|=1} \langle D^{-1/2}(A_G + A_H) D^{-1/2} x, x\rangle %\nonumber 
\\ 
% & = & \max_{x\bot X, \|x\|=1} \langle A D^{-1/2} x, D^{-1/2} x\rangle \nonumber \\ 
 & =  \max_{x\bot X, \|x\|=1} \langle A_G D^{-1/2} x, D^{-1/2} x\rangle + \langle A_H D^{-1/2} x, D^{-1/2} x\rangle\nonumber \\ 
% & \leq & \max_{\substack{x\bot X, \|x\|=1\\ x = D^{1/2}D_G^{-1/2}y}} \langle A_G D_G^{-1/2} y, D_G^{-1/2} y\rangle + 
%  \|A_H D^{-1/2} x\| \|D^{-1/2} x\| \nonumber \\ 
 & \leq  \max_{\substack{x\bot X, \|x\|=1\\ x = D^{1/2}D_G^{-1/2}y}} \langle A_G D_G^{-1/2} y, D_G^{-1/2}  y\rangle +  \|A_H\| \|D^{-1/2}\|^2 \|x\|^2&\nonumber 
	\\ 
 & \leq  \max_{y \bot Y, \|y\| \leq \| D_G^{1/2} D^{-1/2}\|} \frac{\langle D_G^{-1/2} A_G D_G^{-1/2} y, y\rangle}
 {\langle y,y\rangle}\|y\|^2 + \frac{\epsilon}{1-\epsilon} 
%\nonumber 
\\
% & \leq & \max_{\substack{y = D_G^{1/2}D^{-1/2} x \\ x\bot X, \|x\|=1}} \frac{\langle D_G^{-1/2} A_G D_G^{-1/2} y, y\rangle}
% {\langle y,y\rangle} \|D_G^{1/2}\|^2 \|D^{-1/2}\|^2 \|x\|^2 + \frac{\epsilon}{1-\epsilon} \nonumber \\
 & \leq  \max_{y \bot Y}  \frac{\langle D_G^{-1/2} A_G D_G^{-1/2} y, y\rangle}{\langle y,y\rangle} 
 + \frac{\epsilon}{1-\epsilon} \nonumber \\
 & =  \lambda_{n-1}[D_G^{-1/2}A_G D_G^{-1/2}] + \frac{\epsilon}{1-\epsilon}%\nonumber 
\\
 & =  \lambda_{n-1}[I - \mathcal{L}(G)] + \frac{\epsilon}{1-\epsilon} .%\nonumber
\end{align*}

\end{proof}

We also need the fact that in dense random graphs the degree distribution is almost surely concentrated around the average degree.
It is stated in the following well known lemma (which is
a straighforward consequence of Chernoff's inequality).

\begin{lemma}\label{lemma:1}
For every $\epsilon > 0$, there exists a constant $C_{\epsilon}>0$ such that if 
$\rho > C_{\epsilon}{\log m}/{m}$ then a.a.s. for any vertex $v$ 
in $G(m,\rho)$, $$\quad |d(v) - m\rho| < \epsilon m\rho.\quad \qed$$
\end{lemma}

To estimate the spectral gap of $L_i$ 
we use the result by Coja-Oghlan who in 
\cite{C-O2007} gave precise bounds on the eigenvalues of 
the Laplacian of $G(m,\rho)$. 
In particular, he proved the following theorem.

\begin{theorem}\label{th:4}
Let $\mathcal{L}=\mathcal{L}(G(m,\rho))$. There exist constants $c_0, c_1 > 0$ such that if $\rho \geq c_0 {\log m}/{m}$, then a.a.s. 
we have 
$$0 = \lambda_1[\mathcal{L}] < 1-c_1(m\rho)^{-1/2} \leq \lambda_2[\mathcal{L}] \leq \ldots \leq
\lambda_n[\mathcal{L}] \leq 1+ c_1(m\rho)^{-1/2}.$$
\end{theorem}

\begin{proof}[Proof of Theorem \ref{th:3}]
It is enough to show that Theorem~\ref{th:3} 
holds for some $p = C\log n/n^2$, 
where $C > 0$ is a sufficiently large constant.

Note that each graph $L_i$ can be generated in the following 
way. Take an auxiliary  multigraph $\mathcal K$ on $2n$ vertices 
with vertices labeled by the generators and their inverses. 
Two vertices  $a$, $b$, 
where $a\neq b^{-1}$,  are joined by $4(n-1)$ edges in $\mathcal K$, 
while vertices $a$, $a^{-1}$ are joined by $2n-1$ edges. 
This is because, there are $2(n-1)$ irreducible words which start
with $ab^{-1}$, 
and $2(n-1)$ irreducible words which start with $ba^{-1}$; 
on the other hand, there are $2n-1$ irreducible words which 
start with $aa$. Then $L_i$ is obtained from $\mathcal K$ by leaving 
each of its edges independently with probability $p$.
First we shall show that the spectral gap of $L_i$ does not differ
significantly from the spectral gap of the random graph 
$G(2n,\rho)$, in which each two vertices are 
joined by an edge  independently 
with probability $\rho = 1-(1-p)^{4n-4}$.

To this end let $\hat{\mathcal K}$ be obtained 
from $\mathcal K$ by adding 
some `supplementary edges' between vertices $a$ and $a^{-1}$, 
so that each pair of vertices of $\hat{\mathcal K}$
is connected by exactly $4n-4$ edges and let $\hat L_i$ 
be the random (multi)graph obtained from $\hat{\mathcal K}$ by selecting its
edges with probability~$p$. We show that a.a.s. $\hat L_i$ 
contains no  edges with multiplicity larger than two,  
all double  edges of $L_i$ form a matching, and furthermore 
each pair $a, a^{-1}$ is connected by at most one edge. 
 Indeed, the probability 
that some edge has multiplicity three is bounded above by 
$$(2n)^2 (4n)^3  p^3\le O(\log^3 n/n)\to 0\,,$$
while the probability that two double edges share a vertex can be estimated from above by 
$$(2n)(2n)^2 (4n)^4 p^4 \le O(\log^4 n/n)\to 0\,.$$
Finally, the probability that there is a multiple edge connecting 
$a$ and $a^{-1}$ for some generator $a$ is bounded from above by 
$$n (4n)^2  p^2 \le O(\log^2 n/n)\to 0\,.$$

Thus, $L_i$ can be viewed as obtained from $L'_i$ which is a copy of 
$G(2n,\rho)$,  $\rho = 1-(1-p)^{4n-4} \ge C \log n/n$, 
by adding to it  some matching $M_i$ (which takes care of
multiple edges) and subtracting another matching $\bar M_i$ 
(which consists of edges generated from supplementary edges of 
$\hat{\mathcal K}$). 
Since $C$ is large, in particular $C \geq c_0$, we infer from  
Theorem \ref{th:4} that 
a.a.s. all eigenvalues of the Laplacian of $L_i'$ but the smallest one are concentrated around 1. 
In particular, a.a.s.
$$\lambda_2[\mathcal{L}(L_i')]> 1-\epsilon$$
for, say, $\eps=0.01$.
Since $C$ is large from  Lemma \ref{lemma:1} 
a.a.s. for any vertex $v\in L_i'$ we have  
$$|d_{L_i'}(v) - 2n\rho | < 2\epsilon n\rho$$
which means that each $L_i'$ is almost regular.
Since a.a.s. $L_i$ can be obtained from $L_i'$ by adding 
a matching and subtracting another one,  using Lemma \ref{lemma:2} we infer that a.a.s.
$$\lambda_2[\mathcal{L}(L_i)] \geq \lambda_2[\mathcal{L}(L_i')] - \frac{\epsilon}{1-\epsilon}
> 1 - 3\epsilon,$$ 
or equivalently $$\lambda_{2n-1}[I-\mathcal{L}(L_i)] < 3\epsilon.$$ Moreover, a.a.s.
$|d_{L_i}(v) - 2C\log n| < 2\epsilon\cdot 2C\log n$. Therefore graphs $L_i$ are also almost regular and the Laplacian of each 
$L_i$ has a large spectral gap.

Thus, it is enough to show that the sum of three graphs with Laplacians having large spectral gaps is also a graph which 
has Laplacian with a large spectral gap.
Let $A_i$ be the adjacency matrix of the graph $L_i$ and $D_i$ be the corresponding diagonal degree matrix of $L_i$. Then 
$A=A_1+A_2+A_3$ is the adjacency matrix of $L$ and 
$D = D_1 + D_2 + D_3$ is its degree matrix. 

It is also easy to see that 
$$\|D_i^{1/2} D^{-1/2}\| \leq 1.$$

Let $X$ be the eigenvector of $D^{-1/2}AD^{-1/2}$ corresponding to the largest 
eigenvalue $\lambda_{2n}[D^{-1/2}AD^{-1/2}] = 1$ and similarly, for $i=1,2,3$, let $X_i$ be the eigenvector of 
$D_i^{-1/2} A_i D_i^{-1/2}$ corresponding to the largest eigenvalue $\lambda_{2n}[D_i^{-1/2} A_i D_i^{-1/2}] = 1$. 
The entries of vectors $X$ and $X_i$ are square roots of vertex degrees in corresponding graphs and 
since $D_i^{1/2}D^{-1/2}X = X_i$ it follows that if $x\bot X$ and $y = D_i^{1/2}D^{-1/2}x$, then $y\bot X_i$.

We can now estimate $\lambda_{2n-1}[I-\mathcal{L}(L)]$ as follows.

\begin{eqnarray}
\lambda_{2n-1}[I-\mathcal{L}(L)] & = & \lambda_{2n-1}[D^{-1/2}AD^{-1/2}]   \nonumber \\
 & = & \max_{x\bot X, \|x\|=1} \langle D^{-1/2}AD^{-1/2} x, x\rangle \nonumber \\
 & = & \max_{x\bot X, \|x\|=1} \sum_{i=1}^3 \langle A_i D^{-1/2} x, D^{-1/2}x \rangle \nonumber \\
 & \leq & \sum_{i=1}^3 \max_{\substack{x\bot X, \|x\|=1\\x = D^{1/2}D_i^{-1/2}y}} \langle A_i D_i^{-1/2} y, D_i^{-1/2} y \rangle \nonumber \\
 & \leq & \sum_{i=1}^3 \max_{y \bot X_i, \|y\| \leq \|D_i^{1/2} D^{-1/2}\|} 
 \frac{\langle A_i D_i^{-1/2} y, D_i^{-1/2} y \rangle}{\langle y,y\rangle} \|y\|^2 \nonumber \\
 & \leq & \sum_{i=1}^3 \max_{\substack{y\bot X_i}} \frac{\langle D_i^{-1/2} A_i D_i^{-1/2} y, y \rangle}{\langle y,y\rangle} 
  \nonumber \\
 & \leq & \sum_{i=1}^3 \lambda_{2n-1}[D_i^{-1/2} A_i D_i^{-1/2}] \nonumber \\
 & = & \sum_{i=1}^3 \lambda_{2n-1}[I-\mathcal{L}(L_i)] \nonumber \\
 & \leq & 9\epsilon.  \nonumber
\end{eqnarray}

Hence, a.a.s. $\lambda_2[\mathcal{L}(L)] = 1 - \lambda_{2n-1}[I-\mathcal{L}(L)] \geq 1 - 9\epsilon$ 
which can be arbitrarily close to 1. In particular, a.a.s. $\lambda_2[\mathcal{L}(L)] > 1/2$.
\end{proof}

\bibliographystyle{plain}

\end{document}